\documentclass[11pt,twoside]{article}
\usepackage{amsmath}
\usepackage{amsthm}
\makeatletter
\renewenvironment{proof}[1][\proofname]{%
	\par\pushQED{\qed}\normalfont%
	\topsep6\p@\@plus6\p@\relax
	\trivlist\item[\hskip\labelsep\bfseries#1\@addpunct{.}]%
	\ignorespaces
}{%
	\popQED\endtrivlist\@endpefalse
}
\makeatother 
\usepackage{mathtools}
\usepackage[mathscr]{eucal}
\usepackage{amssymb}
\usepackage{graphicx}
\usepackage{enumerate}
\usepackage[authoryear,sort&compress,round,comma]{natbib}
\newtheorem{definition}{Definition}
\usepackage{mathptmx}

\newtheorem{corollary}{Corollary}[section]
\newtheorem{theorem}{Theorem}[section]
\newtheorem{lemma}{Lemma}[section]

\newtheorem{example}{Example}[section]
\numberwithin{equation}{section}

\usepackage{graphicx,type1cm,eso-pic,color}

\makeatletter
\AddToShipoutPicture{
	\setlength{\@tempdimb}{.5\paperwidth}
	\setlength{\@tempdimc}{.5\paperheight}
	\setlength{\unitlength}{1pt}
	\put(\strip@pt\@tempdimb,\strip@pt\@tempdimc)
}
\makeatother

\begin{document}

	\setcounter{page}{1}

	\thispagestyle{empty}
	\markboth{}{}

	\pagestyle{myheadings}
	\markboth{}{ }
	
	\date{}
	
	
	\noindent  
	
	\vspace{.1in}
	
	{\baselineskip 20truept
		
		\begin{center}
			{\Large {\bf On General Weighted Extropy of Extreme Ranked Set Sampling }} \footnote{\noindent
				{\bf $^{\#}$} E-mail: nitin.gupta@maths.iitkgp.ac.in\\
				{\bf * }  corresponding author E-mail: pradeep.maths@kgpian.iitkgp.ac.in}\\
			
		\end{center}
		
		\vspace{.1in}
		
		\begin{center}
			{\large {\bf Pradeep Kumar Sahu* and Nitin Gupta $^{\#}$}}\\
			{\large {\it Department of Mathematics, Indian Institute of Technology Kharagpur, West Bengal 721302, India }}
			\\
		\end{center}
	}
	\vspace{.1in}
	\baselineskip 12truept

	
	\begin{center}
		{\bf \large Abstract}\\
	\end{center}
	The extropy measure, introduced by Lad, Sanfilippo, and Agro in their (2015) paper in Statistical Science, has garnered significant interest over the past years. In this study, we present a novel representation for the weighted extropy within the context of extreme ranked set sampling. Additionally, we offer related findings such as stochastic orders, characterizations, and precise bounds. Our results shed light on the comparison between the weighted extropy of extreme ranked set sampling and its counterpart in simple random sampling.
	\section{Introduction}
	In (1952), McIntyre introduced ranked set sampling (RSS) as a method for estimating mean pasture yields. RSS proves to be a superior sampling strategy for estimating population means when compared to simple random sampling (SRS).
	
	Consider a random variable $X$ with probability density function (pdf) $f$, the cumulative distribution function (cdf) $F$ and the survival function (sf) $\bar F=1-F$. Let $\mathbf{X}_{SRS}=\{X_i:\ i=1,\ldots,n\}$ denote a simple random sample of size $n$ from $X$. Now, focusing on the one-cycle RSS: randomly select $n^2$ units from $X$, and then allocate these units randomly into $n$ sets, each of size $n$. Subsequently, rank the n units within each set based on the variable of interest. Start by selecting the smallest ranked unit from the first set, followed by the second smallest from the second set, and so on, until the $n$th smallest ranked unit is chosen from the last set. Repeat this entire sampling process $m$ times to obtain a sample of size $mn$. This resulting sample is referred to as the RSS from the underlying distribution $F$. Let $\mathbf{X}_{RSS}^{(n)}=\{X_{(i:n)i},\ i=1,\ldots,n\}$ denote the  ranked-set sample, where $X_{(i:n)i}$ represent the $i$-th order statistics from the $i$-th sample with sample size $n$. For further insights into RSS, see the works of Wolfe (2004), Chen, Bai, and Sinha (2004), Al-Nasser (2007), Al-Saleh and Diab (2009), and Raqab and Qiu (2019).

           Samawi, Ahmed and Abu-Dayyeh (1996) have developed extreme ranked set sampling (ERSS) to enhance the efficiency of estimating population parameters compared to SRS, using the same number of units. ERSS achieves this by solely utilizing the minimum and maximum ranked units for n when it is even and, minimum and maximum ranked units along with the median ranked unit when it is odd. In the case where the set size $n$ is even, the measurement of units is conducted on the smallest ranked units from the first $\frac{n}{2}$ sets and the largest ranked units from the last $\frac{n}{2}$ sets. If the set size n is odd, the measurement of units involves the smallest ranked units from the first $\frac{n-1}{2}$ sets, the largest ranked units from the subsequent $\frac{n-1}{2}$ sets and the the last unit corresponds to the median value of the last set in the sample. Akdeniz and Yildiz (2023) investigated the effects of ranking error models on the mean estimators based on RSS and some of its modified methods such as ERSS and percentile ranked set sampling (PRSS) for different distribution, set and cycle size in infinite population, Monte Carlo simulation study is conducted, additionally, the study is supported by real life data and it is observed that, RSS, ERSS and PRSS shows better results than Simple Random Sampling (SRS).
	
	Shannon (1948) initially introduced entropy, a concept with relevance across various disciplines such as information theory, physics, probability and statistics, economics, communication theory, etc. Entropy quantifies the average degree of uncertainty associated with the outcomes of a random experiment.
	
	Entropy quantifies the average degree of uncertainty associated with the outcomes of a random experiment. The differential expression for Shannon entropy is as follows:
	\begin{equation*}
		\label{1eq2}
		H(X) = E\left(- \ln f(X)\right) = - {\int_{-\infty}^{\infty}{f(x) \ln \left( f(x)\right)dx}}.
	\end{equation*} 
	This measure finds application in various contexts, including order statistics and record statistics. Examples of its utilization can be observed in the works of Bratpour et al. (2007), Raqab and Awad (2000, 2001), Zarezadeh and Asadi (2010), Abo-Eleneen (2011), Qiu and Jia (2018a, 2018b), and Tahmasebi et al. (2016). (2016).
	
	In recent times, an alternative measure for assessing uncertainty, known as extropy, has garnered significance. Lad et al. (2015) introduced extropy as the complementary dual of Shannon entropy, defining it as:
	\begin{equation}\label{extropy}
		J(X)=-\frac{1}{2} \int_{-\infty}^{\infty}f^2(x)dx=-\frac{1}{2}E\left(f(X)\right).
	\end{equation}

	In an investigation conducted by Qiu (2017), an exploration of characterization outcomes, lower bounds, monotone properties, and statistical applications pertaining to the extropy of order statistics and record values was carried out. Independently, Balakrishnan et al. (2020) and Bansal and Gupta (2021) introduced the concept of weighted extropy as:
	\begin{equation}\label{wtdextropy}
		J^w(X)=-\frac{1}{2}\int_{-\infty}^{\infty}xf^2(x)dx.
	\end{equation}
	
	In their investigation, Balakrishnan et al. (2020) delved into the characterization outcomes and bounds for weighted iterations of extropy, residual extropy, past extropy, bivariate extropy, and bivariate weighted extropy. Meanwhile, Bansal and Gupta (2021) addressed the outcomes related to weighted extropy and weighted residual extropy concerning order statistics and k-record values. In this context, we present the concept of general weighted extropy (GWE) with a weight function $w_1(x)\geq 0$ as:
	\begin{align*}
		J^{w_1}(X)&=-\frac{1}{2} \int_{-\infty}^{\infty}w_1(x) f^2(x)dx\\
		&=-\frac{1}{2}E(\Lambda_X^{w_1}(U)),
	\end{align*}
	where $\Lambda_X^{w_1}(u)=w_1(F^{-1}(u))f(F^{-1}(u))$ and $U$ is uniformly distributed random variable on $(0,1)$, i.e., $U\sim$ Uniform$(0,1)$.
	
	One may see the example of Bansal and Gupta (2021) and Gupta and Chaudhary (2023) that although the extropies are identical, there is a distinction between weighted extropies and general weighted extropies. Hence, both weighted extropies and general weighted extropies serve as measures of uncertainty. Unlike the extropy defined in (\ref{extropy}), this measure, dependent on the shift, incorporates the values of the random variable.
	
	In their work, Qiu and Raqab (2022) presented a formulation for the weighted extropy associated with ranked set sampling, expressed in relation to quantile and density-quantile functions. They further supplied pertinent outcomes, encompassing monotone properties, stochastic orders, characterizations, and precise bounds. Additionally, the authors demonstrated the comparative analysis between the weighted extropy of ranked set sampling and its equivalent in simple random sampling. In their investigation, Gupta and Chaudhary (2023) explored the monotonic and stochastic characteristics of the general weighted extropy concerning ranked set sampling (RSS) data. The research involved obtaining stochastic comparison results through varied weightings for RSS extropy and included a comparative analysis with simple random sampling (SRS) data. The study also yielded characterization results and an investigation into the monotone properties of the general weighted extropy associated with RSS data.

         In this manuscript, we study the stochastic properties of general weighted extropy of ERSS data. Stochastic comparison results are obtained by taking different weights for the extropy of ERSS. A comparison between the extropy of ERSS and SRS data is provided. Some characterization results are obtained.

	\section{Some results on GWE and related properties}
	Before providing results, let us review definitions from the literature (see, Shaked and Shanthikumar (2007) and Misra, Gupta and Dhariyal (2008)) of some useful terminology.
	
	\begin{definition}
		Let a random variable $X$ have  pdf $f(x)$, cdf $F(x)$ and sf $\bar{F}(x) = 1-F(x).$ Let $l_X = inf\{x \in \mathbb{R} : F(x) > 0\},\   u_X = sup\{x \in \mathbb{R} : F(x) < 1\}$ and	 $S_X= (l_X,u_X),$ where -$\infty \le l_X \le u_X \le \infty$.\\
		\\
		(i) $X$ is said to be log-concave (log-convex) if $\{x \in \mathbb{R} : f(x) > 0\} = S_X$
		and $ln(f(x))$ is concave (convex) on $S _X$.\\
		\\
		(ii) $X$ is said to have increasing (decreasing) failure rate IFR (DFR) if $\bar F(x)$ is log-concave		(log-convex) on  $S_X$.\\
		\\
		(iii) $X$ is said to have decreasing (increasing) reverse failure rate DRFR (IRFR) if $F(x)$ is	log-concave (log-convex) on $S_X$.\\
		\\
		(iv) $X$ is said to have decreasing (increasing) mean residual life DMRL (IMRL) if $\int_{x}^{u_X} \bar F(t)dt$		is log-concave (log-convex) on $S_X$.\\
		\\
		(v) $X$ is said to have increasing (decreasing) mean inactivity time (IMIT (DMIT)) if $\int_{l_X}^{x}	F(t)dt$
		is log-concave (log-convex) on $S_X.$
		
	\end{definition}

	\begin{definition}
		Let $X$ be a random variable with  pdf $f(x)$, cdf $F(x)$ and sf $\bar{F}(x)=1-F(x).$ Let $l_X = inf\{x \in \mathbb{R} : F(x) > 0\}$, $u_X = sup\{x \in \mathbb{R} : F(x) < 1\}$
		and $S_X = (l_X, u_X).$ Similarly, let $Y$ be a random variable with  pdf $g(x)$, cdf $G(x)$ and sf $\bar{G}(x)= 1-G(x).$ Let $l_Y=inf\{x \in R:G(x)> 0\},\  u_Y = sup\{x \in R :G(x) < 1\}$ and $S_Y = (l_Y,u_Y )$. If $l_X \ge 0$ and $l_Y \ge 0$, then \\
		\\
		(i) $X$ is said to be smaller than $Y$ in usual stochastic (st) ordering		$(X \le_{st} Y )$ if $\bar{F}(x) \le \bar{G}(x)$, for every -$\infty< x <\infty.$\\
		\\
		(ii) $X$ is said to be smaller than $Y$ in the likelihood ratio (lr) ordering $(X \le_{lr} Y )$ if
		$g(x)f(y) \le  f(x)g(y)$, whenever $-\infty < x < y <\infty$.\\
		\\
		(iv) $X$ is said to be smaller than $Y$ in the dispersive ordering ($X \le_{disp} Y)$ if $G^{-1}F(x)-x$ is increasing in $x \ge 0$.\\
		\\
		(v) $X$ is said to be smaller than $Y$ in the hazard rate ordering $(X \le_{hr} Y)$ if $\frac{\bar{G}(x)}{\bar{F}(x)}$  is increasing in $x\in S_X \cap S_Y.$
		
	\end{definition}

\section{\textbf{General Weighted Extropies of ERSS}}	
	Let $X$ be a random variable with finite mean $\mu$ and variance $\sigma^2$. For $\textbf{X}_{SRS}=\{X_i,\ i=1,\ldots,n\}$, the joint pdf is $\prod_{i=1}^{n}f(x_i)$, as $X_i$'s, $i=1,\ldots,n$ are independent and identically distributed (i.i.d.). Hence the general weighted extropy of $\textbf{X}_{SRS}^{(n)}$ can be defined as
	\begin{align}\label{SRS1}
		J^{w_1}(\textbf{X}_{SRS}^{(n)})&=\frac{-1}{2}\prod_{i=1}^{n}\left(\int_{-\infty}^{\infty}w_1(x_i)f^2(x_i)dx_i\right)\nonumber \\
		&=\frac{-1}{2}\left(-2J^{w_1}(X)\right)^n\nonumber \\
		&=\frac{-1}{2}\left(E(\Lambda_X^{w_1}(U))\right)^n.
	\end{align}
 Let the pdf of $X_{(i,n)i}$ is
      \[f_{i,n}(x) =\frac{n!}{(i-1)!(n-i)!}F^{i-1}(x) \bar F^{n-i}(x) f(x),\  \infty<x<\infty.\]
      	Now, we can write the GWE of   $\textbf{X}_{ERSS}^{(n)}$  as
\begin{align}	
	& J^{w_1}(\textbf{X}_{ERSS}^{(n)})\nonumber \\ 
 &=\begin{cases}
-\frac{1}{2}\left(\prod_{i=1}^{n/2}\left(-2J^{w_1}(X_{(1:n)i})\right)\right)\left(\prod_{i=\frac{n}{2}+1}^{n}\left(-2J^{w_1}(X_{(n:n)i})\right) \right) \text{ if } \mbox{n\  even},\\
-\frac{1}{2}\left(\prod_{i=1}^{\frac{n-1}{2}}\left(-2J^{w_1}(X_{(1:n)i})\right)\right)\left(\prod_{i=\frac{n+1}{2}}^{n-1}\left(-2J^{w_1}(X_{(n:n)i})\right)\right) \left(-2J^{w_1}(X_{(\frac{n+1}{2}:n)n})\right) \text{ if } \mbox{n\  odd},
\end{cases}\nonumber \\
&=\begin{cases}
-\frac{1}{2}\left(\prod_{i=1}^{n/2}\left(\int_{-\infty}^{\infty}w_1(x)f_{1:n}^2(x)dx)\right)\right)\left(\prod_{i=\frac{n}{2}+1}^{n}\left(\int_{-\infty}^{\infty}w_1(x)f_{n:n}^2(x)dx\right) \right) \text{ if } \mbox{n\  even},\\
-\frac{1}{2}\left(\prod_{i=1}^{\frac{n-1}{2}}\left(\int_{-\infty}^{\infty}w_1(x)f_{1:n}^2(x)dx)\right)\right)\left(\prod_{i=\frac{n+1}{2}}^{n-1}\left(\int_{-\infty}^{\infty}w_1(x)f_{n:n}^2(x)dx)\right)\right)\left(\int_{-\infty}^{\infty}w_1(x)f_{\frac{n+1}{2}:n}^2(x)dx)\right) \text{ if } \mbox{n\  odd},
\end{cases}\nonumber \\
&=\begin{cases} \label{ERSS1}
\frac{Q_{1,n}}{2}\left(E\left(\Lambda_X^{w_1} (B_{1:2n-1})\right)\right)^{n/2}\left(E\left(\Lambda_X^{w_1} (B_{2n-1:2n-1})\right) \right)^{n/2} \text{ if } \mbox{n\  even},
\\
\frac{Q_{2,n}}{2}\left(E\left(\Lambda_X^{w_1} (B_{1:2n-1})\right)\right)^{\frac{n-1}{2}} \left(E\left(\Lambda_X^{w_1} (B_{2n-1:2n-1})\right)\right)^{\frac{n-1}{2}}\left(E\left(\Lambda_X^{w_1} (B_{n:2n-1})\right)\right) \text{ if } \mbox{n\  odd},
\end{cases}
\end{align}
	where \begin{align*}Q_{1,n}&=-\frac{n^{2n}}{(2n-1)^n},\\
		Q_{2,n}&=-\frac{n^{2n}(n!)^2((n-1)!)^{2}}{(2n-1)^n\left((\frac{n-1}{2})!\right)^4(2n-1)!}\end{align*}
	and $B_{2i-1:2n-1}$ is a beta distributed random variable, with parameters $(2i-1)$ and $(2n-2i+1)$, having pdf
 \[\phi_{2i-1:2n-1}(u)=\frac{(2n-1)!}{(2i-2)!(2n-2i)!} u^{2i-2}(1-u)^{2n-2i}, 0<u<1.\] Equation $(\ref{ERSS1})$ provides an expression in simplified form of the GWE	of $\textbf{X}_{ERSS}^{(n)}$. Now we provide some examples to illustrate the equation $(\ref{ERSS1})$.
	
	\begin{example}
		Let $V$ be a random variable with power distribution. The pdf and cdf of $V$ are respectively $f(x)= \theta x^{\theta -1}$ 
		and $F(x)= x^{\theta}$ , $0<x<1$ ,  $\theta > 0$. Let $w_1(x)=x^m,\ x>0, \ m>0$, then it follows that 
        \begin{align*}
            \Lambda_V^{w_1}(u)= w_1(F^{-1}(u)) f(F^{-1}(u))=\theta u^{\frac{m + \theta - 1}{\theta}},
        \end{align*}
        for $ w_1(x)=x^m.$ 
		Then we have
		\begin{align*}	
	& J^{w_1}(\textbf{V}_{ERSS}^{(n)})\nonumber \\ 
 &=\begin{cases}
\frac{Q_{1,n}}{2}\left(\prod_{i=1}^{n/2} E\left(\Lambda_V^{w_1} (B_{1:2n-1})\right)\right)\left(\prod_{i=\frac{n}{2}+1}^{n}E\left(\Lambda_V^{w_1} (B_{2n-1:2n-1})\right) \right) \text{ if } \mbox{n\  even},\\
\frac{Q_{2,n}}{2}\left(\prod_{i=1}^{\frac{n-1}{2}} E\left(\Lambda_V^{w_1} (B_{1:2n-1})\right)\right)\left(\prod_{i=\frac{n+1}{2}}^{n-1}E\left(\Lambda_V^{w_1} (B_{2n-1:2n-1})\right) \right) \left(E\left(\Lambda_V^{w_1} (B_{n:2n-1})\right)\right)\text{ if } \mbox{n\  odd}.
\end{cases}\nonumber \\
 &=\begin{cases}
\frac{Q_{1,n}}{2}\left(\prod_{i=1}^{n/2} \int_{0}^{1}\left(\Lambda_V^{w_1}(u) (2n-1)(1-u)^{(2n-2)} du\right)\right)  \left(\prod_{i=\frac{n}{2}+1}^{n}\int_{0}^{1}\left(\Lambda_V^{w_1}(u) (2n-1) u^{(2n-2)} du\right) \right) \text{ if } \mbox{n\  even},\\
\frac{Q_{2,n}}{2}\left(\prod_{i=1}^{\frac{n-1}{2}} \int_{0}^{1}\left(\Lambda_V^{w_1}(u) (2n-1)(1-u)^{(2n-2)} du\right)\right) \left(\prod_{i=\frac{n+1}{2}}^{n-1}\int_{0}^{1}\left(\Lambda_V^{w_1}(u) (2n-1) u^{(2n-2)} du\right) \right) \nonumber \\ \ \ \ \ \ \ \ \ \ \ \ \ \ \ \ \ \ \ \ \ \ \ \ \ \ \ \ \ \ \ \ \ \ \ \ \ \ \ \left( \int_{0}^{1} \Lambda_V^{w_1}(u) \frac{(2n-1)!}{[(n-1)!]^{2}} u^{(n-1)}(1-u)^{(n-1)} du\right)\text{ if } \mbox{n\  odd}.
\end{cases}\nonumber \\
&=\begin{cases}
\frac{Q_{1,n}}{2}\left(\prod_{i=1}^{n/2} \int_{0}^{1}(2n-1) \theta u^{\frac{m+\theta-1}{\theta}}(1-u)^{(2n-2)}du\right)  \left(\prod_{i=\frac{n}{2}+1}^{n}\int_{0}^{1}(2n-1) \theta u^{\frac{m+\theta-1}{\theta}} u^{(2n-2)} du\right) \text{ if } \mbox{n\  even},\\
\frac{Q_{2,n}}{2}\left(\prod_{i=1}^{\frac{n-1}{2}} \int_{0}^{1}(2n-1) \theta u^{\frac{m+\theta-1}{\theta}}(1-u)^{2n-2}du\right)\left(\prod_{i=\frac{n+1}{2}}^{n-1}\int_{0}^{1}(2n-1) \theta u^{\frac{m+\theta-1}{\theta}} u^{(2n-2)} du\right) \nonumber \\ \ \ \ \ \ \ \ \ \ \ \ \ \ \ \ \ \ \ \ \ \ \ \ \ \ \ \ \ \ \ \ \ \ \ \ \ \ \  \left( \int_{0}^{1} \theta u^{\frac{m+\theta-1}{\theta}}\frac{(2n-1)!}{[(n-1)!]^{2}} u^{n-1}(1-u)^{(n-1)}du\right)\text{ if } \mbox{n\  odd}.
\end{cases}\nonumber\\
&=\begin{cases}
\frac{Q_{1,n}}{2} \theta^{n} \left(\prod_{i=1}^{n/2} \int_{0}^{1}(2n-1)  u^{\frac{m+\theta-1}{\theta}}(1-u)^{(2n-2)}du\right)  \left(\prod_{i=\frac{n}{2}+1}^{n}\int_{0}^{1}(2n-1) u^{\frac{m+\theta-1}{\theta}} u^{(2n-2)} du\right) \text{ if } \mbox{n\  even},\\
\frac{Q_{2,n}}{2} {\theta^{n}}\left(\prod_{i=1}^{\frac{n-1}{2}} \int_{0}^{1}(2n-1)  u^{\frac{m+\theta-1}{\theta}}(1-u)^{2n-2}du\right)\left(\prod_{i=\frac{n+1}{2}}^{n-1}\int_{0}^{1}(2n-1)  u^{\frac{m+\theta-1}{\theta}} u^{(2n-2)} du\right) \nonumber \\ \ \ \ \ \ \ \ \ \ \ \ \ \ \ \ \ \ \ \ \ \ \ \ \ \ \ \ \ \ \ \ \ \ \ \ \ \ \  \left( \int_{0}^{1}  u^{\frac{m+\theta-1}{\theta}}\frac{(2n-1)!}{[(n-1)!]^{2}} u^{n-1}(1-u)^{(n-1)}du\right)\text{ if } \mbox{n\  odd}.
\end{cases}\nonumber\\
&=\begin{cases}
\frac{Q_{1,n}}{2} \theta^{\frac{3n}{2}} (\frac{2n-1}{2n\theta+m-1})^{\frac{n}{2}}[(2n-1)!]^{\frac{n}{2}}\prod_{i=1}^{n/2} \frac{\Gamma{(\frac{2\theta+m-1}{\theta})}}{\Gamma{(\frac{(2n+1)\theta+m-1}{\theta})}}  \text{ if } \mbox{n\  even},\\
\frac{Q_{2,n}}{2} \theta^{\frac{3n-1}{2}} (\frac{2n-1}{2n\theta+m-1})^{\frac{n-1}{2}}\frac{((2n-1)!)^{\frac{n+1}{2}}}{(n-1)!} \frac{\Gamma{(\frac{(n+1)\theta+m-1}{\theta})}}{\Gamma{(\frac{(2n+1)\theta+m-1}{\theta})}} \prod_{i=1}^{n/2} \frac{\Gamma{(\frac{2\theta+m-1}{\theta})}}{\Gamma{(\frac{(2n+1)\theta+m-1}{\theta})}}\text{ if } \mbox{n\  odd}.
\end{cases}
\end{align*} \hfill $\blacksquare$
	\end{example}
 
	\begin{example}
		Let $W$ have an exponential distribution with cdf $F_W(w)=1-e^{-\lambda w}, \ \lambda >0, \ w>0$. Let $w_1(x)=x^m, \ m>0, \ x>0$, then it follows that
		\begin{align*}
	\Lambda_W^{w_1} (u) = w_1(F^{-1}(u)) f(F^{-1}(u))
		= \frac{(-1)^m (1-u) (ln(1-u))^m}{\lambda^{m-1}}, 0<u<1.	
		\end{align*}
		Then we have 
		\begin{align*}
  & J^{w_1}(\textbf{W}_{ERSS}^{(n)})\nonumber \\ 
 &=\begin{cases}
\frac{Q_{1,n}}{2}\left(\prod_{i=1}^{n/2} E\left(\Lambda_V^{w_1} (B_{1:2n-1})\right)\right)\left(\prod_{i=\frac{n}{2}+1}^{n}E\left(\Lambda_V^{w_1} (B_{2n-1:2n-1})\right) \right) \text{ if } \mbox{n\  even},\\
\frac{Q_{2,n}}{2}\left(\prod_{i=1}^{\frac{n-1}{2}} E\left(\Lambda_V^{w_1} (B_{1:2n-1})\right)\right)\left(\prod_{i=\frac{n+1}{2}}^{n-1}E\left(\Lambda_V^{w_1} (B_{2n-1:2n-1})\right) \right) \left(E\left(\Lambda_V^{w_1} (B_{n:2n-1})\right)\right)\text{ if } \mbox{n\  odd}.
\end{cases}\nonumber \\
 &=\begin{cases}
\frac{Q_{1,n}}{2}\left(\prod_{i=1}^{n/2} \int_{0}^{1}\left(\Lambda_V^{w_1}(u) (2n-1)(1-u)^{(2n-2)} du\right)\right)  \left(\prod_{i=\frac{n}{2}+1}^{n}\int_{0}^{1}\left(\Lambda_V^{w_1}(u) (2n-1) u^{(2n-2)} du\right) \right) \text{ if } \mbox{n\  even},\\
\frac{Q_{2,n}}{2}\left(\prod_{i=1}^{\frac{n-1}{2}} \int_{0}^{1}\left(\Lambda_V^{w_1}(u) (2n-1)(1-u)^{(2n-2)} du\right)\right) \left(\prod_{i=\frac{n+1}{2}}^{n-1}\int_{0}^{1}\left(\Lambda_V^{w_1}(u) (2n-1) u^{(2n-2)} du\right) \right) \nonumber \\ \ \ \ \ \ \ \ \ \ \ \ \ \ \ \ \ \ \ \ \ \ \ \ \ \ \ \ \ \ \ \ \ \ \ \ \ \ \ \left( \int_{0}^{1} \Lambda_V^{w_1}(u) \frac{(2n-1)!}{[(n-1)!]^{2}} u^{(n-1)}(1-u)^{(n-1)} du\right)\text{ if } \mbox{n\  odd}.
\end{cases}\nonumber \\
 &=\begin{cases}
\frac{Q_{1,n}}{2}\left(\prod_{i=1}^{n/2} \int_{0}^{1}\frac {(-1)^{m}(1-u)(ln(1-u))^{m}}{\lambda^{m-1}} (2n-1)(1-u)^{(2n-2)} du\right) \left(\prod_{i=\frac{n}{2}+1}^{n}\int_{0}^{1} \frac {(-1)^{m}(1-u)(ln(1-u))^{m}}{\lambda^{m-1}} (2n-1) u^{(2n-2)} du \right) \nonumber \\ \ \ \ \ \ \ \ \ \ \ \ \ \ \ \ \ \ \ \ \ \ \ \ \ \ \ \ \ \ \ \ \ \ \ \ \ \ \ \ \ \ \ \ \ \ \ \ \ \ \ \ \ \  \ \ \ \ \ \ \ \ \  \ \ \ \ \ \ \ \ \ \ \ \ \ \ \ \ \ \ \ \ \ \ \ \ \ \ \ \ \ \ \ \ \ \ \ \ \ \ \ \ \ \ \ \ \ \ \ \ \ \ \ \ \ \  \ \ \ \ \ \ \ \ \ \ \ \ \ \ \ \ \ \ \ \ \ \ \ \ \ \ \ \ \ \ \ \ \ \ \ \ \ \ \ \ \ \text{ if } \mbox{n\  even},\\
\frac{Q_{2,n}}{2}\left(\prod_{i=1}^{\frac{n-1}{2}} \int_{0}^{1}\frac {(-1)^{m}(1-u)(ln(1-u))^{m}}{\lambda^{m-1}} (2n-1)(1-u)^{(2n-2)} du\right) \left(\prod_{i=\frac{n+1}{2}}^{n-1}\int_{0}^{1}\frac {(-1)^{m}(1-u)(ln(1-u))^{m}}{\lambda^{m-1}} (2n-1)(1-u)^{(2n-2)} du\right) \nonumber \\ \ \ \ \ \ \ \ \ \ \ \ \ \ \ \ \ \ \ \ \ \ \ \ \ \ \ \ \ \ \ \ \ \ \ \ \ \ \ \left( \int_{0}^{1} \frac {(-1)^{m}(1-u)(ln(1-u))^{m}}{\lambda^{m-1}} \frac{(2n-1)!}{[(n-1)!]^{2}} u^{(n-1)}(1-u)^{(n-1)} du\right)\text{ if } \mbox{n\  odd}.
\end{cases}\nonumber \\
		\end{align*}
		Taking $u=1-e^{-x}$ in the above equation, we get
		\begin{align*}
			& J^{w_1}(\textbf{W}_{ERSS}^{(n)})\\ 
			&=\begin{cases}
\frac{Q_{1,n}}{2}\left(\prod_{i=1}^{n/2} \frac{(-1)^{m}}{\lambda^{m-1}} \int_{0}^{\infty} e^{-x} (ln( e^{-x}))^{m}(2n-1)(e^{-x})^{(2n-2)} e^{-x} dx\right) \nonumber \\ \ \ \ \ \ \ \ \ \ \ \ \ \ \ \ \ \ \ \ \ \ \ \ \ \ \ \ \ \ \ \ \ \ \ \ \ \ \ \ \ \ \ \ \ \ \ \ \ \ \  \left(\prod_{i=\frac{n}{2}+1}^{n} \frac{(-1)^{m}}{\lambda^{m-1}} \int_{0}^{\infty}e^{-x}(ln(e^{-x}))^{m} (2n-1) (1-e^{-x})^{(2n-2)} e^{-x}dx \right) \text{ if } \mbox{n\  even},\\
\frac{Q_{2,n}}{2}\left(\prod_{i=1}^{\frac{n-1}{2}} \frac{(-1)^{m}}{\lambda^{m-1}} \int_{0}^{\infty} e^{-x} (ln( e^{-x}))^{m}(2n-1)(e^{-x})^{(2n-2)} e^{-x} dx\right) \nonumber \\ \ \ \ \ \ \ \ \ \ \ \ \ \ \ \ \ \ \ \ \ \ \ \ \ \ \ \ \ \ \ \ \ \ \ \ \ \ \left(\prod_{i=\frac{n+1}{2}}^{n-1}\frac{(-1)^{m}}{\lambda^{m-1}} \int_{0}^{\infty}e^{-x}(ln(e^{-x}))^{m} (2n-1) (1-e^{-x})^{(2n-2)} e^{-x}dx \right) \nonumber \\ \ \ \ \ \ \ \ \ \ \ \ \ \ \ \ \ \ \ \ \ \ \ \ \ \ \ \ \ \ \ \ \ \ \ \ \ \ \ \left( \int_{0}^{\infty} \frac {(-1)^{m} e^{-x}(ln(e^{-x}))^{m}}{\lambda^{m-1}} \frac{(2n-1)!}{[(n-1)!]^{2}} (1-e^{-x})^{(n-1)}(e^{-x})^{(n-1)} e^{-x} dx\right)\text{ if } \mbox{n\  odd}.
\end{cases}\nonumber 
\end{align*}
\begin{align*}
&=\begin{cases}
\frac{Q_{1,n}}{2} \frac{1}{\lambda^{n(m-1)}}\left(\prod_{i=1}^{n/2} \int_{0}^{\infty} \frac{2n-1}{2n} x^{m} \frac{(2n)!}{(2n-1)!} (e^{-x})^{2n} dx\right)\left(\prod_{i=\frac{n}{2}+1}^{n} \int_{0}^{\infty}\frac{1}{2n} x^{m} \frac{(2n)!}{(2n-2)!} ( 1-e^{-x})^{2n-2} (e^{-x})^{2}dx \right) \text{ if } \mbox{n\  even},\\
\frac{Q_{2,n}}{2} \frac{1}{\lambda^{n(m-1)}}\left(\prod_{i=1}^{\frac{n-1}{2}} \int_{0}^{\infty} \frac{2n-1}{2n} x^{m} \frac{(2n)!}{(2n-1)!} (e^{-x})^{2n} dx\right) \left(\prod_{i=\frac{n+1}{2}}^{n-1}\int_{0}^{\infty}\frac{1}{2n} x^{m} \frac{(2n)!}{(2n-2)!} ( 1-e^{-x})^{2n-2} (e^{-x})^{2}dx \right) \nonumber \\ \ \ \ \ \ \ \ \ \ \ \ \ \ \ \ \ \ \ \ \ \ \ \ \ \ \ \ \ \ \ \ \ \ \ \ \ \ \ \left( \int_{0}^{\infty} \frac{1}{2} x^{m} \frac {(2n)!}{(n-1)!(n!)} (1-e^{-x})^{n-1} (e^{-x})^{n+1}dx\right)\text{ if } \mbox{n\  odd}.
\end{cases}\nonumber \\
&=\begin{cases}
\frac{Q_{1,n}(2n-1)!!}{2^{n+1}n^{n}} \frac{1}{\lambda^{n(m-1)}}\left(\prod_{i=1}^{n/2} E(W_{1:2n}^{m})\right)\left(\prod_{i=\frac{n}{2}+1}^{n} E(W_{2n-1:2n}^{m}) \right) \text{ if } \mbox{n\  even},\\
\frac{Q_{2,n}(2n-1)!!}{2^{n+1}n^{n-1}} \frac{1}{\lambda^{n(m-1)}}\left(\prod_{i=1}^{\frac{n-1}{2}} E(W_{1:2n}^{m})\right) \left(\prod_{i=\frac{n+1}{2}}^{n-1}E(W_{2n-1:2n}^{m}) \right) \left( E(W_{n:2n}^{m})\right)\text{ if } \mbox{n\  odd}.
\end{cases}\nonumber \\
		\end{align*}
		where $W_{2i-1:2n}$ is the $(2i-1)$-th order statistics of a sample of size $2n$ from exponential distribution having pdf given by 
		$$\psi_{2i-1:2n}=\frac{(2n)!}{(2i-2)!(2n-2i+1)!}(1-e^{-x})^{2i-2}(e^{-x})^{2n-2i+2},\  x\ge 0,$$ 
		and  $(2n-1)!!= (2n-1)^{\frac{n}{2}} $ if $n$ is even and $(2n-1)!!= (2n-1)^{\frac{n-1}{2}} $ if $n$ is odd.\hfill $\blacksquare$
	\end{example}
	
	\begin{example}
		Let $U$ be a Pareto random variable with cdf $F(x)=1-x^{-\alpha},\  \alpha >0, \ x>1 $. Let $w_1(x)=x^m,\ m>0, \ x>0 $, then we get
		\begin{align*}
			\Lambda_U^{w_1} (u) 
			&= w_1(F^{-1}(u)) f(F^{-1}(u))\\
			&=\alpha (1-u)^{\frac{\alpha - m+1}{\alpha}}.		
		\end{align*}
		The weighted extropy of $\textbf{U}_{ERSS}^{(n)}$ is 
		\begin{align*}
			& J^{w_1}(\textbf{U}_{ERSS}^{(n)}) \\
			 &=\begin{cases}
\frac{Q_{1,n}}{2}\left(\prod_{i=1}^{n/2} E\left(\Lambda_V^{w_1} (B_{1:2n-1})\right)\right)\left(\prod_{i=\frac{n}{2}+1}^{n}E\left(\Lambda_V^{w_1} (B_{2n-1:2n-1})\right) \right) \text{ if } \mbox{n\  even},\\
\frac{Q_{2,n}}{2}\left(\prod_{i=1}^{\frac{n-1}{2}} E\left(\Lambda_V^{w_1} (B_{1:2n-1})\right)\right)\left(\prod_{i=\frac{n+1}{2}}^{n-1}E\left(\Lambda_V^{w_1} (B_{2n-1:2n-1})\right) \right) \left(E\left(\Lambda_V^{w_1} (B_{n:2n-1})\right)\right)\text{ if } \mbox{n\  odd}.
\end{cases}\nonumber 
\end{align*}
\begin{align*}
 &=\begin{cases}
\frac{Q_{1,n}}{2}\left(\prod_{i=1}^{n/2} \int_{0}^{1}\left(\Lambda_V^{w_1}(u) (2n-1)(1-u)^{(2n-2)} du\right)\right)  \left(\prod_{i=\frac{n}{2}+1}^{n}\int_{0}^{1}\left(\Lambda_V^{w_1}(u) (2n-1) u^{(2n-2)} du\right) \right) \text{ if } \mbox{n\  even},\\
\frac{Q_{2,n}}{2}\left(\prod_{i=1}^{\frac{n-1}{2}} \int_{0}^{1}\left(\Lambda_V^{w_1}(u) (2n-1)(1-u)^{(2n-2)} du\right)\right) \left(\prod_{i=\frac{n+1}{2}}^{n-1}\int_{0}^{1}\left(\Lambda_V^{w_1}(u) (2n-1) u^{(2n-2)} du\right) \right) \nonumber \\ \ \ \ \ \ \ \ \ \ \ \ \ \ \ \ \ \ \ \ \ \ \ \ \ \ \ \ \ \ \ \ \ \ \ \ \ \ \ \left( \int_{0}^{1} \Lambda_V^{w_1}(u) \frac{(2n-1)!}{[(n-1)!]^{2}} u^{(n-1)}(1-u)^{(n-1)} du\right)\text{ if } \mbox{n\  odd}.
\end{cases}\nonumber \\
&=\begin{cases}
\frac{Q_{1,n}}{2}\left(\prod_{i=1}^{n/2} \int_{0}^{1}\alpha(1-u)^{\frac{\alpha-m+1}{\alpha}} (2n-1)(1-u)^{(2n-2)} du\right)  \left(\prod_{i=\frac{n}{2}+1}^{n}\int_{0}^{1}\alpha(1-u)^{\frac{\alpha-m+1}{\alpha}} (2n-1) u^{(2n-2)} du \right) \text{ if } \mbox{n\  even},\\
\frac{Q_{2,n}}{2}\left(\prod_{i=1}^{\frac{n-1}{2}} \int_{0}^{1}\alpha(1-u)^{\frac{\alpha-m+1}{\alpha}} (2n-1)(1-u)^{(2n-2)} du\right) \left(\prod_{i=\frac{n+1}{2}}^{n-1}\int_{0}^{1}\alpha(1-u)^{\frac{\alpha-m+1}{\alpha}} (2n-1) u^{(2n-2)} du \right) \nonumber \\ \ \ \ \ \ \ \ \ \ \ \ \ \ \ \ \ \ \ \ \ \ \ \ \ \ \ \ \ \ \ \ \ \ \ \ \ \ \ \left( \int_{0}^{1} \alpha(1-u)^{\frac{\alpha-m+1}{\alpha}} \frac{(2n-1)!}{[(n-1)!]^{2}} u^{(n-1)}(1-u)^{(n-1)} du\right)\text{ if } \mbox{n\  odd}.
\end{cases}\nonumber \\
&=\begin{cases}
\frac{Q_{1,n}}{2}\alpha^{\frac{3n}{2}} (\frac{(2n-1)}{(2\alpha-m+1)})^{\frac{n}{2}}\left(\prod_{i=\frac{n}{2}+1}^{n} (2n-1)!\frac{\Gamma(\frac{2\alpha-m+1}{\alpha})}{\Gamma(\frac{(2n+1)\alpha-m+1}{\alpha})}\right) \text{ if } \mbox{n\  even},\\
\frac{Q_{2,n}}{2} \alpha^{\frac{3n}{2}} (\frac{(2n-1)}{(2\alpha-m+1)})^{\frac{n}{2}} \frac{(2n-1)!}{(n-1)!} \left(\prod_{i=\frac{n+1}{2}}^{n-1}(2n-1)!\frac{\Gamma(\frac{2\alpha-m+1}{\alpha})}{\Gamma(\frac{(2n+1)\alpha-m+1}{\alpha})} \right) \left( \frac{\Gamma(\frac{\alpha-m+1}{\alpha}+n)}{\Gamma(\frac{(2n+1)\alpha-m+1}{\alpha})}\right)\text{ if } \mbox{n\  odd}.
\end{cases}\nonumber \\
		\end{align*}\hfill $\blacksquare$

	\end{example}

		The following result gives the conditions under which the GWE will increase (decrease).	
		
		\begin{theorem}\label{new1}
			Let $X$ be a non-negative absolutely continuous random variable with pdf f and cdf F. Assume $\eta(x)$ is an increasing function and  $\frac{w_1(\eta(x))}{\eta^\prime (x)} \leq (\geq) w_1(x)$ and $\eta(0)=0$. If $V=\eta(X)$, then $J^{w_1}(\textbf{X}_{ERSS}^{(n)})\leq (\geq) J^{w_1}(\textbf{V}_{ERSS}^{(n)})$.

		\end{theorem}
		\begin{proof}
			The proof is on similar lines as the proof of theorem 3.1 of Gupta and Chaudhary (2023).
		\end{proof}
		
		\begin{example}
			Let $Z$ have an exponential distribution with cdf $F_Z(z)=1-e^{-\lambda z}, \ \lambda >0, \ z>0$. Let $w_1(x)=x^{2}>0$. Consider $\eta(x)=e^x-1,\ x \geq 0$. Then $\eta (Z)$
			is Pareto distribution (see  Qiu and Raqab (2022) Example 2.7 and  Gupta and Chaudhary (2023) Example 3.4  ) with survival function $\bar F_{\eta (Z)}(x)=1/(1+x)^\lambda$, $x\geq 0$. Note that 
			\begin{align*}
				\frac{w_1(\eta(x))}{\eta^\prime (x)}=\frac{(e^x-1)^2}{e^x}= e^{x}+e^{-x}-2\geq x^{2}=w_1(x).
			\end{align*}
			Hence using Theorem \ref{new1}, $J^{w_1}(\textbf{Z}_{ERSS}^{(n)})\geq J^{w_1}(\eta(\textbf{Z})_{ERSS}^{(n)})$.
		\end{example}

		\noindent A lower bound for the overall weighted extropy of ERSS data is established, relying on the weighted extropy of the SRS data, as demonstrated in the subsequent result.
		
		\begin{theorem}
			Let $X$ be an absolutely continuous random variable with pdf f and cdf F. Then for $n$ even, 
		\end{theorem}
		\begin{align*}
			\frac{ J^{w_1}(\textbf{X}_{ERSS}^{(n)})}{ J^{w_1}(\textbf{X}_{SRS}^{(n)})}\leq n^{2n},
		\end{align*}
           and for $n$ odd, we have 
           \begin{align*}
			\frac{ J^{w_1}(\textbf{X}_{ERSS}^{(n)})}{ J^{w_1}(\textbf{X}_{SRS}^{(n)})}\leq \frac{n^{2n}}{[(n-1)!]^{2}}.
		\end{align*}
		\begin{proof} Let us consider n even
           \begin{align*}
               J^{w_1}(\textbf{X}_{ERSS}^{(n)}) =\frac{-1}{2} \prod_{i=1}^{\frac{n}{2}}\int_{0}^{1} n^{2} \Lambda_{X}^{w}(u)(1-u)^{2n-2}du \prod_{i=\frac{n}{2}+1}^{n}\int_{0}^{1} n^{2} \Lambda_{X}^{w}(u) u^{2n-2}du
           \end{align*}
           for all $0< u< 1,$ $ (1-u)^{2n-2}\leq 1$ and $ u^{2n-2} \leq 1$
           \begin{align*}
              & \geq \frac{-1}{2}\prod_{i=1}^{\frac{n}{2}}\int_{0}^{1} n^{2} \Lambda_{X}^{w}(u) du \prod_{i=\frac{n}{2}+1}^{n}\int_{0}^{1} n^{2} \Lambda_{X}^{w}(u) du\\
              & = \frac{-1}{2}n^{2n}\prod_{i=1}^{\frac{n}{2}}\int_{0}^{1}  \Lambda_{X}^{w}(u) du \prod_{i=\frac{n}{2}+1}^{n}\int_{0}^{1}  \Lambda_{X}^{w}(u) du\\
              & =\frac{-1}{2}n^{2n}(-2J^{w_1}(\textbf{X}))^{\frac{n}{2}}(-2J^{w_1}(\textbf{X}))^{\frac{n}{2}}\\
              & =\frac{-1}{2}n^{2n}(-2J^{w_1}(\textbf{X}))^{n}\\
              & = n^{2n} J^{w_1}(\textbf{X}_{SRS}^{(n)})
           \end{align*}
           This completes the proof for $n$ even.\\
           For $n$ odd, we have \\
           \begin{align*}
               J^{w_1}(\textbf{X}_{ERSS}^{(n)}) =\frac{-1}{2} \prod_{i=1}^{\frac{n-1}{2}}\int_{0}^{1} n^{2} \Lambda_{X}^{w}(u)(1-u)^{2n-2}du \prod_{i=\frac{n+1}{2}}^{n-1}\int_{0}^{1} n^{2} \Lambda_{X}^{w}(u) u^{2n-2}du \int_{0}^{1} \frac{(n!)^{2}}{[(n-1)!]^{4}} \Lambda_{X}^{w}(u) u^{n-1} (1-u)^{n-1} du
           \end{align*}
           for all $0< u< 1,$ $ (1-u)^{2n-2}\leq 1 $, $ u^{2n-2} \leq 1$,  $ (1-u)^{n-1}\leq 1 $, and $ u^{n-1} \leq 1$,
           \begin{align*}
              & \geq \frac{-1}{2}\prod_{i=1}^{\frac{n-1}{2}}\int_{0}^{1} n^{2} \Lambda_{X}^{w}(u) du \prod_{i=\frac{n+1}{2}}^{n-1}\int_{0}^{1} n^{2} \Lambda_{X}^{w}(u) du \int_{0}^{1} \frac{(n!)^{2}}{[(n-1)!]^{4}} \Lambda_{X}^{w}(u) du\\
              & = \frac{-1}{2}n^{2(n-1)} \frac{(n!)^{2}}{[(n-1)!]^{4}} \prod_{i=1}^{\frac{n-1}{2}}\int_{0}^{1}  \Lambda_{X}^{w}(u) du \prod_{i=\frac{n+1}{2}}^{n-1}\int_{0}^{1}  \Lambda_{X}^{w}(u) du \int_{0}^{1} \Lambda_{X}^{w}(u) du \\
              & =\frac{-1}{2}n^{2(n-1)} \frac{(n!)^{2}}{[(n-1)!]^{4}} (-2J^{w_1}(\textbf{X}))^{n}\\
              & = n^{2(n-1)} \frac{(n!)^{2}}{[(n-1)!]^{4}} J^{w_1}(\textbf{X}_{SRS}^{(n)})
           \end{align*}
           This completes the proof for $n$ odd.\hfill $\blacksquare$
		\end{proof}

		\section {Characterization results}

		 The lemma presented by Fashandi and Ahmadi (2012) plays a crucial role as a fundamental tool for constructing different characterizations of symmetric distributions.
         \begin{lemma}
             Let $X$ be a random variable with cdf $F$, pdf $f$, and mean $\mu$.Then $f(\mu+x)=f(\mu-x)$ for all $x \geq 0$ if and only if
             $f(F^{-1}(u))=f(F^{-1}(1-u))$. 
         \end{lemma}
		\begin{theorem}
			Let $X$ be an absolutely continuous random variable with pdf f and cdf F; and assume $w_1(-x)=-w_1(x)$. Then $X$ is a symmetric distributed random variable with mean 0 if and only if	$J^{w_1}(\textbf{X}_{ERSS}^{(n)})=0$  for all odd $n\geq 1$. 
		\end{theorem}
		\begin{proof}
			For sufficiency, suppose $f(x)=f(-x)$ for all $x\geq 0$. Also since $F^{-1}(u)=- F^{-1}(1-u)$,  $f(F^{-1}(u))=f(F^{-1}(1-u))$ for all $0<u<1$ and $w_1(-x)=-w_1(x)$, which implies that
			\begin{equation*}
				\Lambda_X^{w_1} (u)=w_1(F^{-1}(u))f(F^{-1}(u))=-w_1(F^{-1}(1-u)) f(F^{-1}(1-u))=-\Lambda_X^{w_1} (1-u)
			\end{equation*}
			For $n$ odd, we have 
           \begin{align*}
              & J^{w_1}(\textbf{X}_{ERSS}^{(n)})\\
               &=\frac{Q_{2,n}}{2} \prod_{i=1}^{\frac{n-1}{2}}\int_{0}^{1} \Lambda_{X}^{w}(u)\phi_{1:2n-1}(u)du \prod_{i=\frac{n+1}{2}}^{n-1}\int_{0}^{1}  \Lambda_{X}^{w}(u) \phi_{2n-1:2n-1}(u)du  \int_{0}^{1} \Lambda_{X}^{w}(u) \phi_{n:2n-1}(u)du\\
              & = \frac{Q_{2,n}}{2} \left( \prod_{i=1}^{\frac{n-1}{2}} -\int_{0}^{1} \Lambda_{X}^{w}(1-u)(2n-1)(1-u)^{2n-2}du \right) \left(\prod_{i=\frac{n+1}{2}}^{n-1} -\int_{0}^{1}  \Lambda_{X}^{w}(1-u) (2n-1) u^{2n-2}du \right) \\ & \ \ \ \ \ \ \ \ \ \ \left( -\int_{0}^{1} \Lambda_{X}^{w}(1-u) \frac{(2n-1)!}{[(n-1)!]^{2}} u^{n-1} (1-u)^{n-1}du \right)\\
              & =-\frac{Q_{2,n}}{2} \left( \prod_{i=1}^{\frac{n-1}{2}} \int_{0}^{1} \Lambda_{X}^{w}(v)(2n-1)(v)^{2n-2}dv \right) \left(\prod_{i=\frac{n+1}{2}}^{n-1} \int_{0}^{1}  \Lambda_{X}^{w}(v) (2n-1) (1-v)^{2n-2}dv \right) \\ & \ \ \ \ \ \ \ \ \ \  \left( \int_{0}^{1} \Lambda_{X}^{w}(v) \frac{(2n-1)!}{[(n-1)!]^{2}} (1-v)^{n-1} (v)^{n-1}du \right)\\
              & = -\frac{Q_{2,n}}{2} \left( \prod_{i=1}^{\frac{n-1}{2}} \int_{0}^{1} \Lambda_{X}^{w}(v)\phi_{2n-1:2n-1}(v)dv \right) \left(\prod_{i=\frac{n+1}{2}}^{n-1} \int_{0}^{1}  \Lambda_{X}^{w}(v) \phi_{1:2n-1}(v)dvdv \right)  \left( \int_{0}^{1} \Lambda_{X}^{w}(v) \phi_{n:2n-1}(v)dv \right)\\
              & = -  J^{w_1}(\textbf{X}_{ERSS}^{(n)}) 
           \end{align*}
			This completes the proof of sufficiency.\\
              The proof for necessity is on the similar lines as of  Gupta and Chaudhary (2023), hence omitted.\hfill $\blacksquare$
		\end{proof}

            The exponential distribution holds significance in reliability theory and theoretical justifications due to its memoryless property. Exploring equivalent characterizations of the exponential distribution from various perspectives is intriguing. Qiu (2017) and Xiong, Zhuang, and Qiu (2021) demonstrated characterizations based on the extropy of order statistics and record values, respectively. In the following, we will establish that the exponential distribution can also be characterized through the weighted extropy of ERSS data.
            \begin{theorem}
               Let $X$ be a non negative absolutely continuous random variable with pdf $f$ and cdf $F$. Then $X$ has has standard exponential distribution if and only if, for all $n \geq 1$,
                \begin{align*}
			& J^{w_1}(\textbf{X}_{ERSS}^{(n)})\nonumber \\ 
 &=\begin{cases}
\frac{Q_{1,n}(2n-1)!!}{2^{n+1}n^{n}} \frac{1}{\lambda^{n(m-1)}}\left(\prod_{i=1}^{n/2} E(W_{1:2n}^{m})\right)\left(\prod_{i=\frac{n}{2}+1}^{n} E(W_{2n-1:2n}^{m}) \right) \text{ if } \mbox{n\  even},\\
\frac{Q_{2,n}(2n-1)!!}{2^{n+1}n^{n-1}} \frac{1}{\lambda^{n(m-1)}}\left(\prod_{i=1}^{\frac{n-1}{2}} E(W_{1:2n}^{m})\right) \left(\prod_{i=\frac{n+1}{2}}^{n-1}E(W_{2n-1:2n}^{m}) \right) \left( E(W_{n:2n}^{m})\right)\text{ if } \mbox{n\  odd},
\end{cases}\nonumber \\
		\end{align*}
		where $W_{2i-1:2n}$ is the $(2i-1)$-th order statistics of a sample of size $2n$ from exponential distribution having pdf given by 
		$$\psi_{2i-1:2n}=\frac{(2n)!}{(2i-2)!(2n-2i+1)!}(1-e^{-x})^{2i-2}(e^{-x})^{2n-2i+2},\  x\ge 0,$$ 
		and  $(2n-1)!!= (2n-1)^{\frac{n}{2}} $ if $n$ is even and $(2n-1)!!= (2n-1)^{\frac{n-1}{2}} $ if $n$ is odd, as obtained in Example 3.2. 
     \end{theorem}
     \begin{proof}
         The sufficiency follows from Example 3.2. Now we prove the necessity. Assume that equation,
  \begin{align*}
			& J^{w_1}(\textbf{X}_{ERSS}^{(n)})\nonumber \\ 
 &=\begin{cases}
\frac{Q_{1,n}(2n-1)!!}{2^{n+1}n^{n}} \frac{1}{\lambda^{n(m-1)}}\left(\prod_{i=1}^{n/2} E(W_{1:2n}^{m})\right)\left(\prod_{i=\frac{n}{2}+1}^{n} E(W_{2n-1:2n}^{m}) \right) \text{ if } \mbox{n\  even},\\
\frac{Q_{2,n}(2n-1)!!}{2^{n+1}n^{n-1}} \frac{1}{\lambda^{n(m-1)}}\left(\prod_{i=1}^{\frac{n-1}{2}} E(W_{1:2n}^{m})\right) \left(\prod_{i=\frac{n+1}{2}}^{n-1}E(W_{2n-1:2n}^{m}) \right) \left( E(W_{n:2n}^{m})\right)\text{ if } \mbox{n\  odd}.
\end{cases}\nonumber 
		\end{align*}
 holds for all $n\geq 1$. In particular letting $n=1$ and $m=1$ we have,
     $ J^{w_1}(\textbf{X}_{ERSS}^{(1)})= \frac{-1}{8}$ 
     $=-\frac{1}{2}\int_{-\infty}^{\infty}xe^{-2x} dx =J^w(X).$
     Since $ J^{w_1}(\textbf{X}_{ERSS}^{(1)})= J^w(X),$ This  implies that 
     $\int_{0}^{\infty} x (f(x)+e^{-x})(f(x)-e^{-x})dx=0.$
     which means $X$ has standard exponential distribution with $f(x)=e^{-x},\ x \geq 0.$  \hfill $\blacksquare$
    \end{proof}

		\section{Stochastic comparision}
		In the following result, we provide the conditions for comparing two ERSS schemes under different weights.
		
		\begin{theorem}\label{thm com rss1}
			Let $X$  and $Y$ be nonnegative random variables with pdf's $f$ and $g$, cdf's $F$ and $G$, respectively having $u_X=u_Y<\infty$.\\
			(a) If $w_1$ is increasing, $w_1(x)\geq w_2(x)$ and $X\le_{disp} Y$, then $J^{w_1}(\textbf{X}_{ERSS}^{(n)})\le J^{w_2}(\textbf{Y}_{ERSS}^{(n)})$.\\
			(b)   If $w_1$ is increasing, $w_1(x)\leq w_2(x)$ and $X\ge_{disp} Y$, then $J^{w_1}(\textbf{X}_{ERSS}^{(n)})\ge J^{w_2}(\textbf{Y}_{ERSS}^{(n)})$.
		\end{theorem}
		\begin{proof}
			(a) Since $X\le_{disp} Y$, therefore we have $f(F^{-1}(u))\ge g(G^{-1}(u))$ for all $u\in (0,1)$. Then using Theorem 3.B.13(b) of Shaked and Shanthikumar (2007), $X\le_{disp} Y$ implies that $X\ge_{st} Y$. Hence $F^{-1}(u) \ge G^{-1}(u)$ $\forall$ $u\in (0,1)$. Since $w_1$ is increasing and $w_1(x)\geq w_2(x)$, then $w_1(F^{-1}(u)) \ge w_1(G^{-1}(u))\ge w_2(G^{-1}(u))$. 
			Hence 
			\begin{align}\label{stor1}
				\Lambda_X^{w_1} (u)&=w_1(F^{-1}(u)) f(F^{-1}(u))\nonumber\\
				&\ge w_2(G^{-1}(u)) g(G^{-1}(u))\nonumber\\
				=&\Lambda_Y^{w_2} (u).
			\end{align}
			Now using (\ref{stor1}),
			\begin{align*}	
	& J^{w_1}(\textbf{X}_{ERSS}^{(n)})\nonumber \\ 
 &=\begin{cases} 
\frac{Q_{1,n}}{2}\left(E\left(\Lambda_X^{w_1} (B_{1:2n-1})\right)\right)^{n/2}\left(E\left(\Lambda_X^{w_1} (B_{2n-1:2n-1})\right) \right)^{n/2} \text{ if } \mbox{n\  even},
\\
\frac{Q_{2,n}}{2}\left(E\left(\Lambda_X^{w_1} (B_{1:2n-1})\right)\right)^{\frac{n-1}{2}} \left(E\left(\Lambda_X^{w_1} (B_{2n-1:2n-1})\right)\right)^{\frac{n-1}{2}}\left(E\left(\Lambda_X^{w_1} (B_{n:2n-1})\right)\right) \text{ if } \mbox{n\  odd}.
\end{cases}\nonumber \\
& \leq \begin{cases} 
\frac{Q_{1,n}}{2}\left(E\left(\Lambda_Y^{w_1} (B_{1:2n-1})\right)\right)^{n/2}\left(E\left(\Lambda_Y^{w_1} (B_{2n-1:2n-1})\right) \right)^{n/2} \text{ if } \mbox{n\  even},
\\
\frac{Q_{2,n}}{2}\left(E\left(\Lambda_Y^{w_1} (B_{1:2n-1})\right)\right)^{\frac{n-1}{2}} \left(E\left(\Lambda_Y^{w_1} (B_{2n-1:2n-1})\right)\right)^{\frac{n-1}{2}}\left(E\left(\Lambda_Y^{w_1} (B_{n:2n-1})\right)\right) \text{ if } \mbox{n\  odd}.
\end{cases}\nonumber \\
&=  J^{w_1}(\textbf{Y}_{ERSS}^{(n)}).\nonumber
\end{align*}
			(b) Proof is similar to part (a).\hfill $\blacksquare$
		\end{proof}
		
		If we take $w_1(x)=w_2(x)$ in the above theorem, then The following corollary follows.
		
		\begin{corollary}\label{cor com rss1}
			Let $X$  and $Y$ be nonnegative random variables with pdf's $f$ and $g$, cdf's $F$ and $G$, respectively having $u_X=u_Y<\infty$; let $w_1$ be increasing. Then\\
			(a) If $X\le_{disp} Y$, then $J^{w_1}(\textbf{X}_{ERSS}^{(n)})\le J^{w_1}(\textbf{Y}_{ERSS}^{(n)})$.\\
			(b)  If $X\ge_{disp} Y$, then $J^{w_1}(\textbf{X}_{ERSS}^{(n)})\ge J^{w_1}(\textbf{Y}_{ERSS}^{(n)})$.
		\end{corollary}
		

		\begin{lemma}\label{lemma1} [Ahmed et al. (1986); also see Qiu and Raqab (2022), lemma 4.3]
			Let $X$ and $Y$  be nonnegative random variables with pdf's $f$ and $g$, respectively, satisfying $f(0)\ge g(0)>0$. If $X\le_{su}Y$ (or $X\le_{*}Y$ or $X\le_{c}Y$), then $X\le_{disp}Y$.
		\end{lemma}
		
		One may refer Shaked and Shantikumar (2007) for details of convex transform order ($\leq_c$), star order ($\leq_{\star}$), super additive order ($\leq_{su}$), and dispersive order ($\leq_{disp}$). In view of Theorem \ref{thm com rss1}  and Lemma \ref{lemma1}, the following result is obtained.
		\begin{theorem}
			Let $X$  and $Y$ be nonnegative random variables with pdf's $f$ and $g$, cdf's $F$ and $G$, respectively having $u_X=u_Y<\infty$.\\
			(a)  If $w_1$ is increasing, $w_1(x)\geq w_2(x)$ and $X\le_{su} Y$ (or $X\le_{*}Y$ or $X\le_{c}Y$), then $J^{w_1}(\textbf{X}_{ERSS}^{(n)})\le J^{w_2}(\textbf{Y}_{ERSS}^{(n)})$.\\
			(b)   If $w_1$ is increasing, $w_1(x)\leq w_2(x)$ and $X\ge_{su} Y$  (or $X\ge_{*}Y$ or $X\ge_{c}Y$), then $J^{w_1}(\textbf{X}_{ERSS}^{(n)})\ge J^{w_2}(\textbf{Y}_{ERSS}^{(n)})$.
		\end{theorem}
		
		If we take $w_1(x)=w_2(x)$ in the above theorem, then we have the following corollary.
		\begin{corollary}\label{cor com rss2}
			Let $X$  and $Y$ be nonnegative random variables with pdf's $f$ and $g$, cdf's $F$ and $G$, respectively having $u_X=u_Y<\infty$ and  $w_1$ is increasing.\\
			(a) If $X\le_{su} Y$ (or $X\le_{*}Y$ or $X\le_{c}Y$), then $J^{w_1}(\textbf{X}_{ERSS}^{(n)})\le J^{w_1}(\textbf{Y}_{RSS}^{(n)})$.\\
			(b)  If $X\ge_{su} Y$  (or $X\ge_{*}Y$ or $X\ge_{c}Y$), then $J^{w_1}(\textbf{X}_{ERSS}^{(n)})\ge J^{w_1}(\textbf{Y}_{RSS}^{(n)})$.
		\end{corollary}
		
		
		
		\begin{theorem}\label{thm com rss3}
			Let $X$  and $Y$ be nonnegative random variables with pdf's $f$ and $g$, cdf's $F$ and $G$, respectively. Let $\Delta (u)=w_1(F^{-1}(u))f(F^{-1}(u))-w_2(G^{-1}(u))g(G^{-1}(u))$,
			\[A_1=\{0\le u \le 1|\Delta (u)>0\},\ A_2=\{0\le u \le 1|\Delta (u)<0\}.\]
			If $ \inf_{A_1}\phi_{2i-1:2n-2i}(u)\geq  \sup_{A_2}\phi_{2i-1:2n-2i}(u)$, and if $J^{w_1}(X)\leq J^{w_2}(Y)$, then  $J^{w_1}(\textbf{X}_{ERSS}^{(n)})$ $\le J^{w_2}(\textbf{Y}_{ERSS}^{(n)})$.
		\end{theorem}
		\begin{proof}
   The proof is on the similar lines as of Theorem 5.3 of Gupta and Chaudhary (2023), hence omitted.\hfill $\blacksquare$ 
		\end{proof}
		
		If we take $w_1(x)=w_2(x)$ in the above theorem, then we have the following corollary.
		\begin{corollary}\label{cor com rss3}
			Let $X$  and $Y$ be nonnegative random variables with pdf's $f$ and $g$, cdf's $F$ and $G$. Let $\Delta (u)=w_1(F^{-1}(u))f(F^{-1}(u))-w_1(G^{-1}(u))g(G^{-1}(u))$,
			\[A_1=\{0\le u \le 1|\Delta (u)>0\},\ A_2=\{0\le u \le 1|\Delta (u)<0\}.\]
			If $\ \inf_{A_1}\phi_{2i-1:2n-2i}(u) \geq \sup_{A_2}\phi_{2i-1:2n-2i}(u)$ , and if $J^{w_1}(X)\leq J^{w_1}(Y)$, then  $J^{w_1}(\textbf{X}_{ERSS}^{(n)})$ $\le J^{w_1}(\textbf{Y}_{ERSS}^{(n)})$.
		\end{corollary}
		
		
		\begin{example}
			Let $X$  and $Y$ be nonnegative random variables with pdf's $f$ and $g$, respectively. Let
			\begin{eqnarray*}
				f(x)=
				\begin{cases}
					2x, \ 0\leq x<1\\
					0, otherwise
				\end{cases} \ \ ~~~~~~~\text{and}~~~~~~~~~~~~~~~~~
				\ g(x)=
				\begin{cases}
					2(1-x),\  0\leq x<1\\
					0, otherwise.
				\end{cases}
			\end{eqnarray*}
			As pointed in Gupta and Chaudhary (2023), Example 5.1, $\ \inf_{A_1}\phi_{2i-1:2n-2i}(u) = \sup_{A_2}\phi_{2i-1:2n-2i}(u)$ and $J^{w_1}(X) < J^{w_2}(Y)$. Hence using Theorem \ref{thm com rss3}, we have  $J^{w_1}(\textbf{X}_{ERSS}^{(n)})$ $\le J^{w_2}(\textbf{Y}_{ERSS}^{(n)})$. 
		\end{example}

		
         
		\textbf{ \Large Conflict of interest} \\
		\\
		No conflicts of interest are disclosed by the authors.\\
		\\
		\\		
		\textbf{ \Large Funding} \\
		\\
		PKS would like to thank Quality Improvement Program (QIP), All India Council for Technical Education, Government of India (Student Unique Id: FP2200759) for financial assistance. \\

		
		

	\end{document}